\documentclass[11pt, reqno]{amsart}
\usepackage{amssymb,amsthm,amsmath}
\usepackage[numbers,sort&compress]{natbib}
\usepackage{amssymb,amsmath}
\usepackage{amsfonts}
\usepackage{mathrsfs}
\usepackage{latexsym}
\usepackage{amssymb}
\usepackage{amsthm}
\usepackage{pdfsync}
\usepackage{indentfirst}
\usepackage{xcolor}
\newcommand{\R}{{\mathbb R}}

\let\oldsection\section
\renewcommand\section{\setcounter{equation}{0}\oldsection}

\newtheorem{theorem}{Theorem}[section]
\newtheorem{lemma}{Lemma}[section]

\newtheorem{definition}{Definition}[section]

\newtheorem{remark}{Remark}[section]

\let\pa=\partial

\let\f=\frac

\let\Om=\Omega

\def\ba{\begin{eqnarray}}
\def\ea{\end{eqnarray}}

\def\R{\Bbb R}

\def\na{\nabla}

\newcommand{\beq}{\begin{equation}}
\newcommand{\eeq}{\end{equation}}
\newcommand{\ben}{\begin{eqnarray}}
\newcommand{\een}{\end{eqnarray}}
\newcommand{\beno}{\begin{eqnarray*}}
\newcommand{\eeno}{\end{eqnarray*}}

\allowdisplaybreaks

\begin{document}

\title[Scaling invariant Serrin criterion via one velocity
component ]{Scaling invariant Serrin criterion via one velocity
component for the Navier-Stokes equations}

\author{Wendong Wang}
\address[Wendong Wang]{School of Mathematical Sciences, Dalian University of Technology, Dalian, 116024,  P. R. China}
\email{wendong@dlut.edu.cn}

\author{Di Wu}
\address[Di Wu]{School of Mathematical Sciences, Peking University, Beijing, 100871,  P. R. China}
\email{wudi@math.pku.edu.cn}

\author{Zhifei Zhang}
\address[Zhifei Zhang]{School of Mathematical Sciences, Peking University, Beijing, 100871,  P. R. China}
\email{zfzhang@math.pku.edu.cn}

\date{\today}




\begin{abstract}
In this paper, we prove that the Leray weak solution $u : \mathbb{R}^3\times (0, T)\rightarrow\mathbb{R}^3 $ of the Navier-Stokes equations is regular in $\R^3\times (0,T)$ under the scaling invariant Serrin condition imposed on one component
of the velocity $u_3\in L^{q,1}(0, T;L^p(\mathbb{R}^3))$ with
\beno
\frac2q+\frac3p\leq 1,\quad 3<p<+\infty.
\eeno
This result is an immediate consequence of a new local regularity criterion in terms of one velocity component for suitable weak solutions.
\end{abstract}

\maketitle

\allowdisplaybreaks

\section{Introduction}\label{sec1}
In this paper, we study the incompressible Navier-Stokes equations
\begin{equation}\label{eq:NS}
(NS)\left\{\begin{array}{l}
\partial_t u-\Delta u+u\cdot \nabla u+\nabla \pi=0,\\
{\rm div } u=0,\\
u(x,0)=u_0,
\end{array}\right.
\end{equation}
where $\big(u(x,t), \pi(x,t)\big)$ denote the velocity and the pressure of the fluid respectively.

In the pioneering work \cite{Leray}, Leray introduced the concept of weak solutions to $(NS)$ and proved the global existence for datum $u_0\in L^2(\R^3)$. Kato \cite{Kato2} initiated the study of $(NS)$ with initial data belonging to the space $L^3(\R^3)$ and obtained global existence in a subspace of $C([0,\infty),L^3(\R^3))$ provided the norm $\|u_0\|_{L^3(\R^3)}$ is small enough.
 The existence result for initial data small in the Besov space $\dot{B}^{-1+\frac{3}{p}}_{p,q}(\R^3)$ for $p\in [1,\infty)$ and $q\in[1,\infty]$ can be found in \cite{CMP2,gip2}. The function spaces $L^3(\R^3)$ and  $\dot{B}^{-1+\frac{3}{p}}_{p,q}(\R^3)$ for $(p,q)\in [1,\infty)^2$ both guarantee the existence of local-in-time solution for any initial data. Koch and Tataru \cite{KT2} showed that global well-posedness holds as well for small initial data in the space $\mathrm{BMO}^{-1}(\R^3)$. On the other hand, it has been shown by Bourgain and  Pavlovi\'c \cite{BP2} that the Cauchy problem with initial data in $\dot{B}^{-1}_{\infty,\infty}(\R^3)$ is ill-posed no matter how small the initial data is.

In two spatial dimensions, Leray weak solution is unique and regular.
In three spatial dimensions, the regularity and uniqueness of weak solution is an
outstanding open problem in the mathematical fluid mechanics. It was known that if the weak solution $u$ of (\ref{eq:NS}) satisfies
so called Ladyzhenskaya-Prodi-Serrin(LPS) type condition
\beno
\quad u\in L^q(0,T; L^p(\R^3))\quad \textrm{with} \quad \f 2 q+\f 3p\le 1, \quad p\ge 3,
\eeno
then it is regular in $\R^3\times (0,T)$, see \cite{Serrin, Giga, Struwe, ESS}, where the regularity in the class  $L^\infty(0,T; L^3(\R^3))$ was proved
by Escauriaza, Seregin and \v{S}ver\'{a}k \cite{ESS}.  In \cite{gkp12}, based on the work \cite{KK2}, Gallagher, Koch and Planchon gave an alternative proof of the results in \cite{ESS} by the method of profile decomposition. In \cite{gkp2}, they  extended the method in \cite{gkp12} to release the space from $L^3$ to the Besov space with negative power. See \cite{WZ2, Al} for further extensions. Recently, Tao \cite{TT} proved the blow-up rate of the solution $u$ of (\ref{eq:NS}) if the solution $u$ blows up in finite time. We should mention that in the case $\frac{2}{q}+\frac{3}{p}=1$, the function space $L^q_tL^p_x$ is invariant under the Navier-Stokes scaling:
\begin{align}
u(x,t)\mapsto u^{\lambda}(x,t)=\lambda u(\lambda x,\lambda^2 t),\quad\forall \lambda>0,
\end{align}
where $u^\lambda$ is still a solution to \eqref{eq:NS} with initial data $u^\lambda_0:=\lambda u_0(\lambda x)$.

Concerning the partial regularity of weak solution satisfying the local energy inequality, initiated by Scheffer \cite{SV2},
Caffarelli, Kohn and Nirenberg \cite{CKN} showed that one dimensional Hausdorff measure of the possible singular set is zero.
One could check Lin \cite{Lin} and Ladyzhenskaya and Seregin \cite{LS} for the simplified proof and improvements. More generalizations could be found from \cite{TX, GKT, Vasseur, Ku, WZ1, WZ3} and the references therein.

Starting in \cite{NP}, there are many interesting works devoted to a new LPS type criterion, which only involves one component of the velocity. Neustupa,  Novotn\'y and Penel \cite{NNP} proved the LPS type criterion for one component  $u_3\in L^q(0,T;L^p(\mathbb{R}^3))$ with $\frac{2}{q}+\frac{3}{p}\leq \frac{1}{2}$.
Later, this condition was improved by Kukavica and Ziane \cite{KZ} to
\begin{align*}
\frac{2}{q}+\frac{3}{p}= \f 58, \quad p>\frac{24}{5},\quad \f {16} 5\le q<+\infty;
\end{align*}
and by Cao and Titi \cite{CT1} to
\begin{align*}
\frac{2}{q}+\frac{3}{p}\leq \frac{2}{3}+\frac{2}{3p},\quad p>\frac{7}{2};
\end{align*}
and then by Pokorn\'y and  Zhou \cite{PZ} up to
\begin{align*}
\frac{2}{q}+\frac{3}{p}\leq \frac{3}{4}+\frac{1}{2p},\quad  p>\frac{10}{3}.
\end{align*}
However, these conditions are not scaling invariant. Recently, Chemin and Zhang \cite{CZ} obtained a blow-up criterion via one velocity component in a scaling invariant space $L^p_t(\dot{H}_x^{\frac{1}{2}+\frac{2}{p}})$ with $4<p<6$. Later, Chemin,  Zhang and  Zhang \cite{CZZ} released the restriction on $p$ to $4<p<\infty$ and Han et al.  \cite{HLLZ} extended the arrange of $p$ to  $2\le p<+\infty$.
However, as stated in \cite{Neu}, {\it the question whether the condition  $u_3\in L^q(0,T;L^p(\mathbb{R}^3))$ for $p$ and $q$, basically satisfying the condition $\f 2 q + \f 3 p\le1$, is sufficient for regularity of solution $u$ in $\R^3\times (t_1, t_2)$, is still open.}

Very recently, Chae and Wolf \cite{CW} made an important progress and obtained the regularity of solution to \eqref{eq:NS} under the almost LPS type condition
\begin{align*}
u_3\in L^q(0,T;L^p(\mathbb{R}^3)),\quad\frac{2}{q}+\frac{3}{p}<1,\quad 3<p<\infty.
\end{align*}

\smallskip

The aim of this paper is to obtain LPS criterion via one velocity component with $\frac{2}{q}+\frac{3}{p}\le 1$.
Now let us state our main result.

\begin{theorem}\label{thm:main1}
Let $u_0\in L^2(\mathbb{R}^3)\cap L^3(\mathbb{R}^3)$ and $(u,\pi)$ be Leray weak solution of (\ref{eq:NS}) in $\R^3\times (0,T)$.
If $u$ satisfies the following condition
\begin{align}\label{eq:LPS}
u_3\in L^{q,1}(0,T;L^p(\mathbb{R}^3)),\quad\frac{2}{q}+\frac{3}{p}\le 1,\quad 3<p<\infty,
\end{align} then $u$ is a regular in $\R^3\times (0,T)$. Here $L^{q,1}$ denotes the Lorentz space with respect to variable $t$.
\end{theorem}

Theorem \ref{thm:main1} is a consequence of the following Theorem \ref{thm:main2} via a similar compactness argument in \cite{CW}. Hence we omit the detail.

\begin{remark}
The initial data in $L^2(\mathbb{R}^3)\cap L^3(\mathbb{R}^3)$ in Theorem \ref{thm:main1} implies the local-in-time regularity of weak solutions, thus the weak solution is actually suitable weak solution.
\end{remark}

Next let us introduce the definition of suitable weak solution.
\begin{definition} Let $\Om\subset \R^3$ and $T>0$. We say that $(u,\pi)$ is a suitable weak solution
of (\ref{eq:NS}) in $\Om_T=\Om\times (-T,0)$ if
\begin{enumerate}

\item $u\in L^{\infty}(-T,0;L^2(\Om))\cap L^2(-T,0;H^1(\Om))$ and $\pi\in L^{\frac32}(\Om_T)$;

\item \eqref{eq:NS} is satisfied in the sense of distribution;

\item there holds the local energy inequality: for  any nonnegative $\phi\in C_c^\infty(\R^3\times\R)$
vanishing in a neighborhood of the parabolic boundary of $\Om_T$,
\ben\label{eq:energy inequality}
&&\int_{\Om}|u(x,t)|^2\phi dx+2\int_{-T}^t\int_{\Om}|\nabla u|^2\phi dxds\nonumber\\
&&\leq
\int_{-T}^t\int_{\Om}|u|^2(\partial_s\phi+\triangle\phi)+u\cdot\nabla\phi(|u|^2+2\pi)dxds
\een
for any $t\in [-T,0]$.
\end{enumerate}
\end{definition}

\begin{theorem}\label{thm:main2}
Let $(u,\pi)$ be a suitable weak solution of (\ref{eq:NS}) in $\R^3\times (-1,0)$.
If $u$ satisfies  the following  condition
\beno
u_3\in L^{q,1}(-1,0;L^p(\mathbb{R}^3)),\quad \frac2q+\frac3p\le 1,\quad 3< p<\infty,
\eeno
then it holds that
\beno
r^{-2}\|u\|_{L^3(Q_r(z_0))}^3\leq C,
\eeno
for any $0<r<\frac12$ and any $z_0\in\mathbb{R}^3\times(-\frac12,0]$. Here $Q_r=B_r(0)\times (-r^2,0)$.

\end{theorem}

\begin{remark}
Compared with the result in \cite{CW}, our main contribution is that the condition \eqref{eq:LPS} with the equality is invariant under the Navier-Stokes scaling. Due to the inclusion $L^{q,1}\subsetneq L^q$ for $q>1$, the regularity of weak solution under the condition $u_3\in L^q(0,T;L^p(\mathbb{R}^3))$ with $\frac{2}{q}+\frac{3}{p}=1$ is still open.

\end{remark}

%

\section{An intuitive argument}

In this section, let us present an intuitive argument to show the regularity condition via one velocity component.
We introduce
\beno
U(t,x_3)=\int_{\R^2}|u(x_h, x_3,t)|^2dx_h.
\eeno
Then $U(t,x_3)$ satisfies
\beno
\f 12\pa_tU-\f 1 2 \pa_{x_3}^2U+\int_{\R^2}|\na u(x,t)|^2dx_h=-\int_{\R^2}u\cdot \na u\cdot u dx_h-\int_{\R^2}\na \pi\cdot udx_h.
\eeno
By integration by parts and $\na \cdot u=0$, we get
\begin{align*}
&\f 12\pa_tU-\f 1 2 \pa_{x_3}^2U+\int_{\R^2}|\na u(x,t)|^2dx_h
\\&=\int_{\R^2}-\pa_{x_3}u_3\f 12|u|^2dx_h-\int_{\R^2}u_3\pa_{x_3}\f 12|u|^2dx_h-\int_{\R^2}\pa_{x_3}\pi u_3dx_h\\
\\&\qquad-\int_{\R^2}\pi\pa_{x_3}u_3dx_h\\
&=-\pa_{x_3}\int_{\R^2}u_3\f 12|u|^2dx_h-\pa_{x_3}\int_{\R^2}\pi u_3dx_h.
\end{align*}
Since there is a velocity component $u_3$ for each nonlinear term on the right hand side, this simple argument shows the reason that the regularity criterion
via one component is reasonable.

Next let us motivate our result via the following toy equation
\beno
\pa_tU-\pa_{x_3}^2U=-\pa_{x_3}\int_{\R^2}u_3|u|^2dx_h.
\eeno
Then we have
\beno
U(t,x_3)=e^{t\pa_{x_3}^2}U_0-\int_0^te^{(t-s)\pa_{x_3}^2}\pa_{x_3}\Big(\int_{\R^2}u_3|u|^2dx_h\Big)ds.
\eeno
Using the estimate of heat kernel, we obtain
\beno
\|U(t)\|_{L^\infty}\le \|U(t-\delta)\|_{L^\infty}+C\int_{t-\delta}^t(t-s)^{-\f12}\|u_3(s)\|_{L^\infty}\|U(s)\|_{L^\infty}ds,
\eeno
which gives
\beno
\sup_{s\in [t-\delta,t]}\|U(s)\|_{L^\infty}\le \|U(t-\delta)\|_{L^\infty}
+C\int_{t-\delta}^t(t-s)^{-\f12}\|u_3(s)\|_{L^\infty}ds\sup_{s\in [t-\delta,t]}\|U(s)\|_{L^\infty}.
\eeno
Therefore, if $u_3\in L^{2,1}((0,T); L^\infty)$, then we have
\begin{align*}
\int_{t-\delta}^t(t-s)^{-\f12}\|u_3(s)\|_{L^\infty}ds&\le \|(t-s)^{-\f12}\|_{L^{2,\infty}}\big\|\|u_3(s)\|_{L^\infty}\big\|_{L^{2,1}}\\
&\le C\big\|\|u_3(s)\|_{L^\infty}\big\|_{L^{2,1}(t-\delta,t)}.
\end{align*}
This shows that
\beno
\sup_{s\in [t-\delta,t]}\|U(s)\|_{L^\infty}\le \|U(t-\delta)\|_{L^\infty}
+C\big\|\|u_3(s)\|_{L^\infty}\big\|_{L^{2,1}(t-\delta,t)}\sup_{s\in [t-\delta,t]}\|U(s)\|_{L^\infty}.
\eeno
Thus, if $\delta$ is small enough, we conclude that
\beno
\sup_{s\in [t-\delta,t]}\|U(s)\|_{L^\infty}\le 2\|U(t-\delta)\|_{L^\infty}.
\eeno
In particular, this argument implies that
$u\in L^\infty_{t,x_3}(L^2_{x_h})$, which is a scaling invariant estimate  under the Navier-Stokes scaling.

Regarding to nonlocal pressure, this argument seems difficult to apply the original Navier-Stokes equations.

\section{Local energy estimates}

In this section, we apply a similar argument as in the proof of Caffarelli-Kohn-Nirenberg theorem \cite{CKN} or Chae-Wolf \cite{CW},
inserting $\phi=\Phi_n\zeta$ in (\ref{eq:energy inequality}), where $\zeta$ denotes a cut-off function, while $\Phi_n$ stands for the
shifted fundamental solution to the backward heat equation in one spatial dimension,
i.e.,
\beno
\Phi_n(x,t)=\frac{1}{\sqrt{4\pi(-t+r_n^2)}}e^{-\frac{x_3^2}{4(-t+r_n^2)}}, \quad (x,t)\in \mathbb{R}^3\times (-\infty,0),
\eeno
where
\beno
r_n=2^{-n}, \quad n\in \mathbb{N}.
\eeno

Following  the notations in \cite{CW}, for $0<R<\infty$, we set
\beno
&&U_n(R)=U(0,R,r_n)=B'(R)\times (-r_n,r_n),\\
&& Q_n(R)=U_n\times (-r_n^2,0),\\
&& A_n(R)=B'(R)\times A^*_n,
\eeno
where
\beno
A_n^*=Q_n^*\setminus{Q_{n+1}^*}, \quad Q_n^*=(-r_n,r_n)\times (-r_n^2,0).
\eeno
Clearly, there exist absolute constants $c_1, c_2 > 0$ such that for all $0 < R < +\infty$, $n\in \mathbb{N}$
and $j = 1,\cdots, n$, it holds
\ben
&&\Phi_n\leq c_2 r_j^{-1},\quad |\partial_3\Phi_n|\leq c_2 r_j^{-2}\quad {\rm in}~A_j(R);\label{eq:Phi-n Aj}\\
&&c_1 r_n^{-1}\leq \Phi_n\leq c_2 r_n^{-1},\quad  |\partial_3\Phi_n|\leq c_2 r_n^{-2}\quad {\rm in}~Q_n(R)\label{eq:Phi-n Qn}.
\een

Given $R>0$ and $n\in \mathbb{N}_0$,    the following notation will be used in what follows:
\beno
E_n(R)=\displaystyle\sup_{t\in (-r_n^2,0)}\int_{U_n(R)}|u(t)|^2dx+\int_{-r_n^2}^0\int_{U_n(R)}|\nabla u|^2dxds,
\eeno
\beno
\mathcal{E}=\sup_{t\in (-1,0)}\int_{\mathbb{R}^3}|u(t)|^2dx+\int_{-1}^0\int_{\mathbb{R}^3}|\nabla u|^2dxds.
\eeno

%

By Sobolev's embedding theorem and a standard interpolation argument, it follows that
\ben\label{eq:sobolev ml}
\|u\|_{L^m(-r_n^2,0; L^l(U_n(R)))}^2\leq C E_n(R), \quad \forall~2\leq m\leq \infty,\quad \frac{2}{m}+\frac3l=\frac32.
\een
(see also  Lemma 3.1 in \cite{CW}).

Let $\eta(x_3,t)\in C_c^\infty((-1,1)\times(-1,0])$ denote a cut-off function such that $0\leq \eta\leq 1$ in $\mathbb{R}\times (-1,0]$, $\eta=1$ on
$Q_1^*(-\frac12,\frac12)\times (-\frac14,0)$. In addition, let $\frac12\leq \rho<R$ be arbitrarily chosen, but $|R-\rho|<\frac12$. Let $\psi=\psi(x')\in C^\infty(\mathbb{R}^2)$ with $0\leq \psi\leq 1$ in $B'(R)$ satisfying
\begin{equation}\label{eq:psi}\psi(x)=\psi(|x|)=
\left\{\begin{array}{llll}
1\qquad ~{\rm in}~B'(\rho);\\
0 \qquad ~{\rm in} ~\mathbb{R}^2\backslash B'{(\frac{R+\rho}{2})},
\end{array}\right.
\end{equation}
and
\beno
|D\psi| \leq \frac{C}{R-\rho},\quad |D^2\psi| \leq \frac{C}{(R-\rho)^2}.
\eeno

Direct energy estimates yield that
\ben\label{eq:energy w}
&&\frac12\int_{U_0(R)}|u(\cdot,t)|^2\Phi_n(\cdot,t)\eta(\cdot,t)\psi dx+\int_{-1}^t \int_{U_0(R)}|\nabla u|^2\Phi_n\eta\psi dxds\nonumber\\
&&\leq  \frac12\int_{-1}^t\int_{U_0(R)}|u|^2(\partial_t+\triangle)(\Phi_n\eta\psi) dxds+\frac12 \int_{-1}^t\int_{U_0(R)}|u|^2u\cdot\nabla(\Phi_n\eta\psi) dxds\nonumber\\
&&\qquad+\int_{-1}^t\int_{U_0(R)}\pi u\cdot\nabla(\Phi_n\eta\psi)dxds.
\een

Our goal of rest of this section is to control the three terms on the right hand side of the above estimate.

\subsection{Estimates for nonlinear terms}
\begin{lemma}\label{prop:I1I2}
Let $(u,\pi)$ be a suitable weak solution of (\ref{eq:NS}) in $\R^3\times (-1,0)$. Suppose that  $(u,\pi)$ satisfies the same assumption of Theorem \ref{thm:main2}. Then we have
\begin{align}\label{eq:I1}
\int_{-1}^t\int_{U_0(R)}|u|^2(\partial_t+\triangle)(\Phi_n\eta\psi) dxds\leq C\frac{\mathcal{E}}{(R-\rho)^2},
\end{align}
and there exists a positive series $\{B_i\}_{i\in\mathbb{N}}$ with $\sum_{i=0}^{\infty}B_i\leq \|u_3\|_{L^{q,1}_tL^p_x}$ such that for any $n\in\mathbb{N}$ we have
\begin{eqnarray}\label{eq:I2}
\begin{split}
&\int_{-1}^t\int_{U_0(R)}|u|^2u\cdot\nabla(\Phi_n\eta\psi) dxds\\
&\leq C\sum_{i=0}^{n}(r_i^{-1}E_i(R))B_i
+C(R-\rho)^{-1}\mathcal{E}^{\frac12} \sum_{i=0}^{n}r_i^{\frac12}(r_i^{-1}E_i(R))+C\mathcal{E}^{\frac{3}{2}}.
\end{split}
\end{eqnarray}
\end{lemma}
\begin{proof}
Let $(u,\pi)$ be the solution satisfying the condition in Lemma \ref{prop:I1I2}. The proof of \eqref{eq:I1} is quite similar with \cite{CW}. In details, we notice that
\begin{align*}
\int_{-1}^t\int_{U_0(R)}|u|^2(\partial_t+\triangle)&(\Phi_n\eta\psi) dxds\\
&=\int_{-1}^t\int_{U_0(R)}|u|^2(\Phi_n\partial_t\eta\psi+2\partial_3\Phi_n\partial_3\eta\psi+\Phi_n\triangle(\eta\psi)) dxds,
\end{align*}
which along with  \eqref{eq:Phi-n Aj} and \eqref{eq:Phi-n Qn} implies that
\begin{align*}
\int_{-1}^t\int_{U_0(R)}|u|^2&(\partial_t+\triangle)(\Phi_n\eta\psi) dxds\\
&\leq C\int_{A_0(R)}|u|^2dxds+C\frac{1}{(R-\rho)^2}\int_{Q_0(R)}|u|^2\Phi_ndxds\\
&\leq C(R-\rho)^{-2}\sum_{j=0}^nr_j^{-1}\int_{Q_j(R)}|u|^2dxds.
\end{align*}
By taking the $L^\infty$ in time and the definition of $Q_j(R)$,
we obtain
\begin{align*}
\int_{-1}^t\int_{U_0(R)}|u|^2&(\partial_t+\triangle)(\Phi_n\eta\psi) dxds\\
&\leq C(R-\rho)^{-2}\sum_{j=0}^nr_j\mathcal{E}\leq C(R-\rho)^{-2}\mathcal{E}.
\end{align*}
This proves  \eqref{eq:I1}.\smallskip

Next we turn to give the proof of \eqref{eq:I2}. We first notice that
\begin{align*}
&\int_{-1}^t\int_{U_0(R)}|u|^2u\cdot\nabla(\Phi_n\eta\psi) dxds\\
&\leq  \sum_{i=0}^{n-1}\int_{A_i(R)}|u|^2|u_3||\partial_3\Phi_n|\eta\psi dxds+\int_{Q_n(R)}|u|^2|u_3||\partial_3\Phi_n|\eta\psi dxds\\
&\quad+ \int_{Q_0(R)}|u|^3|\nabla\psi|\Phi_n\eta dxds+\int_{Q_0(R)}|u|^2|u_3||\partial_3\eta|\Phi_n\psi dxds\\
&\leq  C\sum_{i=0}^{n-1}\int_{A_i(R)}|u|^2|u_3|\frac{|x_3|}{(\sqrt{(-s+r_n^2)})^3}e^{-\frac{x_3^2}{4(-s+r_n^2)}}\eta\psi dxds\\
&\quad+C\int_{Q_n(R)}|u|^2|u_3|\frac{|x_3|}{(\sqrt{(-s+r_n^2)})^3}e^{-\frac{x_3^2}{4(-s+r_n^2)}}\eta\psi dxds\\
&\quad+ \int_{Q_0(R)}|u|^3|\nabla\psi|\Phi_n\eta dxds+\int_{Q_0(R)}|u|^2|u_3||\partial_3\eta|\Phi_n\psi dxds\doteq I_{21}+\cdots+I_{24}
\end{align*}
We first notice that the last term $I_{24}$ on the right hand side can by easily controlled as
\begin{align}\label{eq:I-24}
I_{24}\leq\int_{Q_0(R)}|u|^2|u_3|dxds\leq C\mathcal{E}^{\frac{3}{2}}.
\end{align}
Before presenting the details of estimate about $I_{21}$ and $I_{22}$, we introduce $B_i$ as
\begin{align*}
B_i=& \sum_{k=i}^{\infty}\left(\int_{-r_k^2}^{0}\|u_3(\cdot,s)\|_{L^p}^{\frac{2p}{2p-3}}ds\right)^{\frac{2p-3}{2p}}r_{k}^{-1}e^{-\frac{r_i^2}{32r_k^2}}\\
&+r_i^{-1}\left(\int_{-r_{i}^2}^{0}\|u_3(\cdot,s)\|_{L^p}^{\frac{2p}{2p-3}} ds\right)^{\frac{2p-3}{2p}}.
\end{align*}

About $I_{21}$, we first notice that
\begin{align*}
I_{21}\leq &C
\sum_{i=0}^{n-1}\int_{A_i(R)\cap\{r_{i+1}\leq |x_3|\leq r_i, -r_{i+1}^2\leq s\leq 0\}}|u|^2|u_3|\frac{|x_3|}{(\sqrt{(-s+r_n^2)})^3}e^{-\frac{x_3^2}{4(-s+r_n^2)}}\eta\psi dxds\\
&+C\sum_{i=0}^{n-1}\int_{A_i(R)\cap \{|x_3|\leq r_i, -r_{i}^2\leq s\leq -r_{i+1}^2\}}|u|^2|u_3|\frac{|x_3|}{(\sqrt{(-s+r_n^2)})^3}e^{-\frac{x_3^2}{4(-s+r_n^2)}}\eta\psi dxds\\
=&I_{211}+I_{212}.
\end{align*}
Firstly, we observe that by \eqref{eq:sobolev ml}
\beno
\|u\|_{L^{\frac{4p}{3}}(-r_i^2,0; L^{2p'}(U_i(R)))}^2\leq C E_i(R),
\eeno
as $\frac{3}{2p}+\frac3{2p'}=\frac32$.
For the first term $I_{211}$, we observe that
\begin{align*}
I_{211}\leq& C\sum_{i=0}^{n-1}\int_{-r_{i+1}^2}^0\|u(\cdot,t)\|_{L^{2p'}(U_i(R))}^2\|u_3(\cdot,s)\|_{L^p}\frac{1}{(-s+r_n^2)}e^{-\frac{r_i^2}{32(-s+r_n^2)}} ds\\
\leq& C\sum_{i=0}^{n-1}E_i(R)\left(\int_{-r_{i+1}^2}^0\|u_3(\cdot,s)\|_{L^p}^{\frac{2p}{2p-3}}\frac{1}{(-s+r_n^2)^{\frac{2p}{2p-3}}}e^{-\frac{pr_i^2}{(32p-12)(-s+r_n^2)}} ds\right)^{\frac{2p-3}{2p}}.
\end{align*}
On the other hand, we notice the following fact that for any $s\in[-r^2_{i+1},0]$
\beno
\frac{1}{(-s+r_n^2)^{\frac{p}{2p-3}}}e^{-\frac{pr_i^2}{(32p-12)(-s+r_n^2)}+\frac{r_i^2}{32(-s+r_n^2)}}&\leq&Cr_i^{-\frac{2p}{2p-3}}.
\eeno
Gathering the above two estimates, we obtain
\begin{align*}
I_{211}\leq& C\sum_{i=0}^{n-1}(r_i^{-1}E_i(R))\left(\int_{-r_{i+1}^2}^0\|u_3(\cdot,s)\|_{L^p}^{\frac{2p}{2p-3}}\frac{1}{(-s+r_n^2)^{\frac{p}{2p-3}}}e^{-\frac{r_i^2}{32(-s+r_n^2)}} ds\right)^{\frac{2p-3}{2p}}.
\end{align*}
 Before going further, we pay our attention to the term:
\beno
\left(\int_{-r_{i+1}^2}^0\|u_3(\cdot,s)\|_{L^p}^{\frac{2p}{2p-3}}\frac{1}{(-s+r_n^2)^{\frac{p}{2p-3}}}e^{-\frac{r_i^2}{32(-s+r_n^2)}} ds\right)^{\frac{2p-3}{2p}},
\eeno
which is actually controlled by $B_i$. Indeed,  we notice that for any $0\leq i\leq n-1$,
\ben\label{eq:Bi1}
&&\left(\int_{-r_{i+1}^2}^0\|u_3(\cdot,s)\|_{L^p}^{\frac{2p}{2p-3}}\frac{1}{(-s+r_n^2)^{\frac{p}{2p-3}}}e^{-\frac{r_i^2}{32(-s+r_n^2)}} ds\right)^{\frac{2p-3}{2p}}\nonumber\\
&&=\left(\int_{-r_{i+1}^2-r^2_n}^{-r_n^2}\|u_3(\cdot,s+r_n^2)\|_{L^p}^{\frac{2p}{2p-3}}\frac{1}{(-s)^{\frac{p}{2p-3}}}e^{-\frac{r_i^2}{-32s}} ds\right)^{\frac{2p-3}{2p}}\nonumber\\
&
&\leq \left(\int_{-r_{i}^2}^{0}\|\chi_n(s)u_3(\cdot,s+r_n^2)\|_{L^p}^{\frac{2p}{2p-3}}\frac{1}{(-s)^{\frac{p}{2p-3}}}e^{-\frac{r_i^2}{-32s}}ds\right)^{\frac{2p-3}{2p}},
\een
where $\chi_n(s)=1-\mathbf{1}_{(-r_n^2,0]}(s)$. Now we denote $J_k=(-r_{k}^2,-r_{k+1}^2]$, then
\begin{eqnarray}\label{eq:bi2}
\begin{split}
&\left(\int_{-r_{i}^2}^{0}\|\chi_n(s)u_3(\cdot,s+r_n^2)\|_{L^p}^{\frac{2p}{2p-3}}\frac{1}{(-s)^{\frac{p}{2p-3}}}e^{-\frac{r_i^2}{-32s}}ds\right)^{\frac{2p-3}{2p}}\\
&\leq\sum_{k=i}^{\infty}\left(\int_{J_k}\|\chi_n(s)u_3(\cdot,s+r_n^2)\|_{L^p}^{\frac{2p}{2p-3}}\frac{1}{(-s)^{\frac{p}{2p-3}}}e^{-\frac{r_i^2}{-32s}}ds\right)^{\frac{2p-3}{2p}}\\
&\leq\sum_{k=i}^{\infty}\left(\int_{J_k}\|\chi_n(s)u_3(\cdot,s+r_n^2)\|_{L^p}^{\frac{2p}{2p-3}}ds\right)^{\frac{2p-3}{2p}}r_{k+1}^{-1}e^{-\frac{r_i^2}{32r_k^2}}.
\end{split}
\end{eqnarray}
On the other hand, we notice that for any $k\geq n$,
\begin{align*}
\int_{J_k}\|\chi_n(s)u_3(\cdot,s+r_n^2)\|_{L^p}^{\frac{2p}{2p-3}}ds=0,
\end{align*}
and for any $i\leq k\leq n-1$,
\begin{align*}
\int_{J_k}\|\chi_n(s)&u_3(\cdot,s+r_n^2)\|_{L^p}^{\frac{2p}{2p-3}}ds=\int_{-r_k^2+r_n^2}^{-r_{k+1}^2+r_n^2}\|u_3(\cdot,s)\|_{L^p}^{\frac{2p}{2p-3}}ds\leq \int_{-r_k^2}^{0}\|u_3(\cdot,s)\|_{L^p}^{\frac{2p}{2p-3}}ds,
\end{align*}
which along with \eqref{eq:Bi1} and \eqref{eq:bi2} implies
\begin{align*}
&\left(\int_{-r_{i+1}^2}^0\|u_3(\cdot,s)\|_{L^p}^{\frac{2p}{2p-3}}\frac{1}{(-s+r_n^2)^{\frac{p}{2p-3}}}e^{-\frac{r_i^2}{32(-s+r_n^2)}} ds\right)^{\frac{2p-3}{2p}}\\
&\leq 2\sum_{k=i}^{\infty}\left(\int_{-r_k^2}^{0}\|u_3(\cdot,s)\|_{L^p}^{\frac{2p}{2p-3}}ds\right)^{\frac{2p-3}{2p}}r_{k}^{-1}e^{-\frac{r_i^2}{32r_k^2}}\leq CB_i.
\end{align*}
Therefore, we obtain
\begin{eqnarray}\label{eq:I211}
\begin{split}
I_{211}\leq  C\sum_{i=0}^{n-1}(r_i^{-1}E_i(R))B_i.
\end{split}
\end{eqnarray}

Similarly, for the second term $I_{212}$, we have
\begin{align*}
I_{212}\leq& C\sum_{i=0}^{n-1}\left(\int_{-r_{i}^2}^{-r_{i+1}^2}\|u_3(\cdot,s)\|_{L^p}^{\frac{2p}{2p-3}}\frac{1}{(-s+r_n^2)^{\frac{2p}{2p-3}}} ds\right)^{\frac{2p-3}{2p}}\|u\|^2_{L^\frac{4p}{3}(-r_{i}^2,-r_{i+1}^2; L^{2p'}(U_i(R)) )}\\
\leq&C\sum_{i=0}^{n-1}(r_i^{-1}E_i(R))\left(\int_{-r_{i}^2}^{-r_{i+1}^2}\|u_3(\cdot,s)\|_{L^p}^{\frac{2p}{2p-3}}\frac{r_i^{\frac{2p}{2p-3}}}{(-s+r_n^2)^{\frac{2p}{2p-3}}} ds\right)^{\frac{2p-3}{2p}}.
\end{align*}
Due to $s\in[-r_i^2,-r^2_{i+1}]$, we then have
\beno
\frac{r_i^{\frac{2p}{2p-3}}}{(-s+r_n^2)^{\frac{2p}{2p-3}}}\leq C\frac{1}{(-s+r^2_n)^{\frac{p}{2p-3}}},
\eeno
which along the above estimate about $I_{212}$ implies
\begin{eqnarray}\label{eq:I212}
\begin{split}
I_{212}\leq&C\sum_{i=0}^{n-1}(r_i^{-1}E_i(R))\left(\int_{-r_{i}^2}^{-r_{i+1}^2}\|u_3(\cdot,s)\|_{L^p}^{\frac{2p}{2p-3}}\frac{1}{(-s+r_n^2)^{\frac{p}{2p-3}}} ds\right)^{\frac{2p-3}{2p}}\\
\leq& C\sum_{i=0}^{n-1}(r_i^{-1}E_i(R)) r_i^{-1}\left(\int_{-r_{i}^2}^{0}\|u_3(\cdot,s)\|_{L^p}^{\frac{2p}{2p-3}}ds\right)^{\frac{2p-3}{2p}}\leq C\sum_{i=0}^{n-1}(r_i^{-1}E_i(R))B_i.
\end{split}
\end{eqnarray}
Therefore, we obtain
\begin{align}\label{eq:I-21}
I_{21}\leq C\sum_{i=0}^{n-1}(r_i^{-1}E_i(R))B_i.
\end{align}
Similarly, we have
\begin{eqnarray}\label{eq:I-22}
\begin{split}
I_{22}&=\int_{Q_n(R)}|u|^2|u_3||\partial_3\Phi_n|\eta\psi dxds\\
&\leq Cr_n^{-1}E_n(R)\left(\int_{-r_n^2}^0\|u_3(\cdot,s)\|_{L^p}^{\frac{2p}{2p-3}}\frac{1}{(\sqrt{-s+r_n^2})^{\frac{2p}{2p-3}}}ds\right)^{\frac{2p-3}{2p}}\\
&\leq Cr_n^{-1}E_n(R)B_n.
\end{split}
\end{eqnarray}

About $I_{23}$, we have
\begin{eqnarray}\label{eq:I-23}
\begin{split}
I_{23}&=\int_{Q_0(R)}|u|^3|\nabla\psi|\Phi_n\eta dxds\leq C\frac{1}{(R-\rho)}\sum_{i=0}^{n}r_i^{-1}\int_{Q_i(R)}|u|^3dxds\\
&\leq C\frac{1}{(R-\rho)}\sum_{i=0}^{n}r_i^{-1}r^{1/2}_i\|u\|_{L^4(-r_i^2,0; L^3(U_i(R)))}^3\\
&\leq C(R-\rho)^{-1}\mathcal{E}^{\frac12} \sum_{i=0}^{n}r_i^{\frac12}(r_i^{-1}E_i(R)).
\end{split}
\end{eqnarray}
Combining \eqref{eq:I-21} to \eqref{eq:I-23}, we finally have
\begin{eqnarray}
\begin{split}
&\int_{-1}^t\int_{U_0(R)}|u|^2u\cdot\nabla(\Phi_n\eta\psi) dxds\\
&\leq C\sum_{i=0}^{n}(r_i^{-1}E_i(R))B_i+C(R-\rho)^{-1}\mathcal{E}^{\frac12} \sum_{i=0}^{n}r_i^{\frac12}(r_i^{-1}E_i(R))+C\mathcal{E}^{\frac{3}{2}}.
\end{split}
\end{eqnarray}
This finishes the proof of \eqref{eq:I2}. We are left with the proof of 
$$
\sum_{i=0}^{\infty}B_i\leq \|u_3(\cdot,s)\|_{L^{q,1}_tL^p_x}.
$$
 We notice that
\begin{eqnarray}\label{eq:sum-B-i}
\begin{split}
\sum_{i=0}^{\infty}B_i\leq& \sum_{i=0}^{\infty}\sum_{k=i}^{\infty}\left(\int_{-r_k^2}^{0}\|u_3(\cdot,s)\|_{L^p}^{\frac{2p}{2p-3}}ds\right)^{\frac{2p-3}{2p}}r_{k}^{-1}e^{-\frac{r_i^2}{32r_k^2}}\\
&+\sum_{i=0}^{\infty}r_i^{-1}\left(\int_{-r_{i}^2}^{0}\|u_3(\cdot,s)\|_{L^p}^{\frac{2p}{2p-3}} ds\right)^{\frac{2p-3}{2p}}.
\end{split}
\end{eqnarray}
About the first term on the right hand side, we have
\beno
&&\sum_{i=0}^{\infty}\sum_{k=i}^{\infty}\left(\int_{-r_k^2}^{0}\|u_3(\cdot,s)\|_{L^p}^{\frac{2p}{2p-3}}ds\right)^{\frac{2p-3}{2p}}r_{k}^{-1}e^{-\frac{r_i^2}{32r_k^2}}\\
&&\leq \sum_{k=0}^{\infty}\left(\int_{-r_k^2}^0\|u_3(\cdot,s)\|_{L^p}^{\frac{2p}{2p-3}}ds\right)^{\frac{2p-3}{2p}}\sum_{i=0}^{k}\frac{1}{r_k}e^{-\frac{r_i^2}{32r_k^2}} \\
&&\leq C\sum_{k=0}^{\infty}\frac{1}{r_k}\left(\int_{-r_k^2}^0\|u_3(\cdot,s)\|_{L^p}^{\frac{2p}{2p-3}}ds\right)^{\frac{2p-3}{2p}}\leq C\sum_{k=0}^{\infty}{r_k}^{1-\frac3p-\frac{2}{\tilde{q}}}\left(\int_{-r_k^2}^0\|u_3(\cdot,s)\|_{L^p}^{\tilde{q}}ds\right)^{\frac{1}{\tilde{q}}},
\eeno
where
\beno
\frac{2p}{2p-3}\leq \tilde{q}<q=\frac{2p}{p-3}, \quad 3<p<\infty.
\eeno
By Lemma \ref{lem:lorentz}, we obtain
\begin{align}
\sum_{i=0}^{\infty}\sum_{k=i}^{\infty}\left(\int_{-r_k^2}^{0}\|u_3(\cdot,s)\|_{L^p}^{\frac{2p}{2p-3}}ds\right)^{\frac{2p-3}{2p}}r_{k}^{-1}e^{-\frac{r_i^2}{32r_k^2}}
\leq C\|u_3(\cdot,s)\|_{L^{q,1}_tL^p_x}.
\end{align}
On the other hand, the control of the second term of \eqref{eq:sum-B-i} is quite obvious, since
\beno
\sum_{i=0}^{\infty}r_i^{-1}\left(\int_{-r_{i}^2}^{0}\|u_3(\cdot,s)\|_{L^p}^{\frac{2p}{2p-3}} ds\right)^{\frac{2p-3}{2p}}
\leq C\|u_3(\cdot,s)\|_{L^{q,1}_tL^p_x}
\eeno
by using H\"{o}lder inequality and Lemma \ref{lem:lorentz} again.

Hence, we obtain
\beno
\sum_{i=0}^{\infty}B_i\leq C\|u_3(\cdot,s)\|_{L^{q,1}_tL^p_x}.
\eeno
The proof of this lemma is completed.
\end{proof}

\subsection{Estimate for the pressure}
This part is devoted to show the estimates about the third term on the right side of  \eqref{eq:energy w}, which is related to the control of the pressure $\pi$. We first decompose the pressure $\pi$ as $\pi=\pi_0+\pi_h$, where
\beno
-\triangle\pi_0=\partial_i\partial_j(u_iu_j\chi_{Q_0(R)})\quad \text{in}\quad\mathbb{R}^3\times (-1,0).
\eeno
Hence $\pi_h$ is harmonic in $Q_0(R)$.
Then we have
\beno
 &&\int_{-1}^t\int_{U_0(R)}\pi u\cdot \nabla(\Phi_n\eta\psi) dxds\\
 &&=\int_{-1}^t\int_{U_0(R)}\pi_0 u\cdot \nabla(\Phi_n\eta\psi) dxds+\int_{-1}^t\int_{U_0(R)}\pi_h u\cdot \nabla(\Phi_n\eta\psi) dxds\\
 &&=\int_{-1}^t\int_{U_0(R)}\pi_0 u_3 \partial_3\Phi_n(\eta\psi) dxds+\int_{-1}^t\int_{U_0(R)}\pi_0  \Phi_n u\cdot \nabla (\eta\psi) dxds\\
 &&\quad-\int_{-1}^t\int_{U_0(R)}\nabla \pi_h \cdot u(\Phi_n\eta\psi) dxds.
\eeno
The purpose of rest of this part is to show the controls of the three terms on the right side of the above equation.

\begin{lemma}\label{prop:pi-01}
Let $(u,\pi)$ be a suitable weak solution of (\ref{eq:NS}) in $\R^3\times (-1,0)$. Suppose that  $(u,\pi)$ satisfies the same assumption of Theorem \ref{thm:main2}. Then  there exists a positive series $\{C_i\}_{i\in\mathbb{N}}$ with $\sum_{i=0}^{\infty}C_i\leq \|u_3(\cdot,s)\|_{L^{q,1}_tL^p_x}$ such that for any $n\in\mathbb{N}$ we have
\begin{align*}
\int_{-1}^t\int_{U_0(R)}u_3\pi_0(\partial_3\Phi_n\eta\psi) dxds\leq C\sum_{i=0}^n (r_i^{-1} E_i(R))C_i.
\end{align*}
\end{lemma}
We can also  represent $\pi_0$ in the following way. For any $f_{ij}\in L^p(Q_0(R))$ with $1<p<\infty$ and $i,j=1,2,3$, we define
\begin{align*}
T(f)(x,t)=\text{P.V.}\int_{\mathbb{R}^3} K(x-y):f(y,t)\chi_{U_0(R)}(y)dy,\quad(x,t)\in\mathbb{R}^3\times(-1,0),
\end{align*}
with the kernel
\begin{align*}
K_{ij}=\partial_i\partial_j\big(\frac{1}{4\pi|x|}\big),\quad i,j=1,2,3.
\end{align*}
Then we have $\pi_0=T(u_iu_j\chi_{Q_0(R)})$.
\begin{proof}
Let $(u,\pi)$ be the solution satisfying the condition in Lemma \ref{prop:pi-01}. We first introduce the following notations and definition.
For  $j\in \mathbb{N}_0$ let $\chi_j=\chi_{Q_j(R)}.$
Moreover, we set
\beno
\phi_j=\left\{\begin{array}{lll}
\chi_{j}-\chi_{j+1},\quad {\rm if }\quad j=0,1,\cdots,n-1;\\
\chi_{n},\quad {\rm if }\quad j=n.
\end{array}\right.
\eeno
It is clear that
\beno
\sum_{j=0}^{n}\phi_j=(\chi_0-\chi_1)+\cdots+\chi_{n}=1 \Rightarrow f=\sum_{j=0}^{n}f\phi_j \quad {\rm in }~Q_0(R).
\eeno
Taking $f=u_iu_j\chi_{Q_0(R)}$, it holds that
\beno
\pi_0=T(f)=\sum_{j=0}^{n}T(f\phi_j)=\sum_{j=0}^{n}\pi_{0,j}.
\eeno
Then we have
\begin{align*}
&\int_{-1}^t\int_{U_0(R)}u_3\pi_0(\partial_3\Phi_n\eta\psi) dxds\\
&= \sum_{k=0}^{n}\int_{-1}^t\int_{U_0(R)}\pi_0 u_3\partial_3\Phi_n \phi_k\eta \psi dxds= \sum_{j=0}^{n}\sum_{k=0}^{n}\int_{-1}^t\int_{U_0(R)}\pi_{0,j} u_3\partial_3\Phi_n \phi_k\eta \psi dxds\\
&=\sum_{k=0}^{n}\sum_{j=k}^{n}\int_{-1}^t\int_{U_0(R)}\pi_{0,j} u_3\partial_3\Phi_n \phi_k\eta\psi dxds+\sum_{j=0}^{n}\sum_{k=j+1}^{n}\int_{-1}^t\int_{U_0(R)}\pi_{0,j} u_3\partial_3\Phi_n \phi_k\eta \psi dxds\\
&=II_1+II_2.
\end{align*}
We now deal with the term $II_1$. By the definitions of cut-off function $\phi_i$ and singular operator $T$, $II_1$ can be written as
\beno
II_1=\sum_{k=0}^{n}\int_{-1}^t\int_{U_0(R)}\Pi_{0,k} u_3\partial_3\Phi_n \phi_k\eta\psi dxds
\eeno
with
\beno
\Pi_{0,k}=\left\{\begin{array}{lll}
\pi_0,\quad {\rm if }\quad k=0;\\
T(\chi_{k} f),\quad {\rm if }\quad k=1,\cdots,n.
\end{array}\right.
\eeno
We first notice that
\begin{align*}
II_1\leq &C
\sum_{i=0}^{n-1}\int_{A_i(R)\cap\{r_{i+1}\leq |x_3|\leq r_i, -r_{i+1}^2\leq s\leq 0\}}|\Pi_{0,i}||u_3|\frac{|x_3|}{(\sqrt{(-s+r_n^2)})^3}e^{-\frac{x_3^2}{4(-s+r_n^2)}}\eta dxds\\
&+C\sum_{i=0}^{n-1}\int_{A_i(R)\cap \{|x_3|\leq r_i, -r_{i}^2\leq s\leq -r_{i+1}^2\}}|\Pi_{0,i}||u_3|\frac{|x_3|}{(\sqrt{(-s+r_n^2)})^3}e^{-\frac{x_3^2}{4(-s+r_n^2)}}\eta dxds\\
&+C\int_{Q_n(R)}|\Pi_{0,n}||u_3|r_n^{-2} dxds\\
\leq &C\sum_{i=0}^{n-1}\int_{-r_{i+1}^2}^0\|u(\cdot,t)\|_{L^{2p'}(U_i(R))}^2\|u_3(\cdot,s)\|_{L^p}\frac{1}{(-s+r_n^2)}e^{-\frac{r_i^2}{32(-s+r_n^2)}} ds\\
&+C\sum_{i=0}^{n-1}r_i^{-2}\left(\int_{-r_{i}^2}^{-r_{i+1}^2}\|u_3(\cdot,s)\|_{L^p}^{\frac{2p}{2p-3}} ds\right)^{\frac{2p-3}{2p}}\|u\|^2_{L^\frac{4p}{3}(-r_{i}^2,-r_{i+1}^2; L^{2p'}(U_i(R)) )}\\
&+Cr_n^{-2}\left(\int_{-r_{n}^2}^{0}\|u_3(\cdot,s)\|_{L^p}^{\frac{2p}{2p-3}} ds\right)^{\frac{2p-3}{2p}}\|u\|^2_{L^\frac{4p}{3}(-r_{n}^2,0; L^{2p'}(U_n(R)) )}
\end{align*}
provided that the kernel of $T$ is a Calder\'on-Zygmund kernel such that for any $t\in(-1,0)$,
\begin{align*}
\|\Pi_{0,k}(\cdot,t)\|_{L^{P'}(\mathbb{R}^3)}\leq C\|u_iu_j(\cdot,t)\chi_{k}(\cdot,t)\|_{L^{p'}}\leq \|u(\cdot,t)\|^2_{L^{2p'}}.
\end{align*}
By a similar argument leading to \eqref{eq:I211} and \eqref{eq:I212}, we obtain
\begin{eqnarray}\label{eq:pressure-II1}
\begin{split}
II_1\leq&C\sum_{i=0}^{n}(r_i^{-1}E_i(R))B_i.
\end{split}
\end{eqnarray}

Now we turn to show the control the $II_2$, which is much more complicated. We first have
\begin{align*}
II_2= &\sum_{j=n-2}^{n}\sum_{k=j}^{n}\int_{-1}^t\int_{U_0(R)}\pi_{0,j} u_3\partial_3\Phi_n \phi_k\eta\psi dxds\\
&+\sum_{j=0}^{n-3}\sum_{k=j}^{j+3}\int_{-1}^t\int_{U_0(R)}\pi_{0,j} u_3\partial_3\Phi_n\phi_k\eta\psi dxds\\
&+\sum_{j=0}^{n-3}\sum_{k=j+4}^{n}\int_{-1}^t\int_{U_0(R)}\pi_{0,j} u_3\partial_3\Phi_n\phi_k\eta\psi dxds=II_{21}+II_{22}+II_{23}.
\end{align*}
Using the property of singular operator $T$ and a similar argument as above, we get
\begin{align*}
II_{21}\leq &C\sum_{i=n-2}^{n}\sum_{k=i}^{n}\int_{Q_k(R)}|\pi_{0,i}||u_3|\frac{|x_3|}{(\sqrt{(-s+r_n^2)})^3}e^{-\frac{x_3^2}{4(-s+r_n^2)}}\eta\psi dxds\\
\leq &C\sum_{i=n-2}^{n}\sum_{k=i}^{n}\left(\int_{-r_{k}^2}^{0}\|u_3(\cdot,s)\|_{L^p}^{\frac{2p}{2p-3}} ds\right)^{\frac{2p-3}{2p}}r_n^{-2}\|u\|^2_{L^\frac{4p}{3}(-r_{k}^2,0; L^{2p'}(U_i(R)) )}\\
\leq& C\sum_{i=0}^{n}(r_i^{-1}E_i(R))B_i,
\end{align*}
and
\begin{align*}
II_{22}\leq &C
\sum_{i=0}^{n-3}\sum_{k=i}^{i+3}\int_{A_k(R)\cap\{r_{k+1}\leq |x_3|\leq r_k, -r_{k+1}^2\leq s\leq 0\}}|\pi_{0,i}||u_3|\frac{|x_3|}{(\sqrt{(-s+r_n^2)})^3}e^{-\frac{x_3^2}{4(-s+r_n^2)}}\eta\psi dxds\\
&+C\sum_{i=0}^{n-3}\sum_{k=i}^{i+3}\int_{A_k(R)\cap \{|x_3|\leq r_k, -r_{k}^2\leq s\leq -r_{k+1}^2\}}|\pi_{0,i}||u_3|\frac{|x_3|}{(\sqrt{(-s+r_n^2)})^3}e^{-\frac{x_3^2}{4(-s+r_n^2)}}\eta\psi dxds\\
\leq &C\sum_{i=0}^{n-3}\sum_{k=i}^{i+3}\int_{-r_{k+1}^2}^0\|u(\cdot,t)\|_{L^{2p'}(U_i(R))}^2\|u_3(\cdot,s)\|_{L^p}\frac{1}{(-s+r_n^2)}e^{-\frac{r_k^2}{32(-s+r_n^2)}} ds\\
&+C\sum_{i=0}^{n-3}\sum_{k=i}^{i+3}\left(\int_{-r_{k}^2}^{-r_{k+1}^2}\|u_3(\cdot,s)\|_{L^p}^{\frac{2p}{2p-3}}\frac{1}{(-s)^{\frac{2p}{2p-3}}} ds\right)^{\frac{2p-3}{2p}}\|u\|^2_{L^\frac{4p}{3}(-r_{k}^2,-r_{k+1}^2; L^{2p'}(U_i(R)) )}\\
\leq& C\sum_{i=0}^{n}(r_i^{-1}E_i(R))\left(\int_{-r_{i+1}^2}^0\|u_3(\cdot,s)\|_{L^p}^{\frac{2p}{2p-3}}\frac{1}{(-s+r_n^2)^{\frac{p}{2p-3}}}e^{-\frac{r_i^2}{32(-s+r_n^2)}} ds\right)^{\frac{2p-3}{2p}}\\
&+C\sum_{i=0}^{n}(r_i^{-1}E_i(R))\left(\int_{-r_{i}^2}^{-r_{i+1}^2}\|u_3(\cdot,s)\|_{L^p}^{\frac{2p}{2p-3}}\frac{1}{(-s)^{\frac{p}{2p-3}}} ds\right)^{\frac{2p-3}{2p}}
\end{align*}
Hence, as in \eqref{eq:I211} and \eqref{eq:I212} we obtain
\begin{eqnarray}\label{eq:pressure-II2122}
\begin{split}
II_{21}+II_{22}\leq C\sum_{i=0}^{n}(r_i^{-1}E_i(R))B_i.
\end{split}
\end{eqnarray}
At last, we estimate the term $II_{23}$ as
\beno
II_{23}=\sum_{j=0}^{n-3}\sum_{k=j+4}^{n}\int_{-1}^t\int_{U_0(R)}\pi_{0,j} u_3\partial_3\Phi_n \phi_k\eta\psi dxds.
\eeno
By the definition of $\pi_{0,j}$, which is harmonic in $\mathbb{R}^2\times (-r_{j+2},r_{j+2})\times (-r_{j+2}^2,0)$, with the help of Lemma  \ref{lem:harmonic estimate} we get
\beno
\|\pi_{0,j}(\cdot,s)\|_{L^{p'}(U_k(R))}\leq C r_k^{\frac{1}{p'}}r_j^{\frac{2}{p'}-\frac{3}{\ell}}\|\pi_{0,j}(\cdot,s)\|_{L^{\ell}(\mathbb{R}^3)}.
\eeno
Hence, it follows that
\begin{align*}
II_{23}\leq &C
\sum_{i=0}^{n-3}\sum_{k=i+4}^{n-1}\int_{A_k(R)\cap\{r_{k+1}\leq |x_3|\leq r_k, -r_{k+1}^2\leq s\leq 0\}}|\pi_{0,i}||u_3|\frac{|x_3|}{(\sqrt{(-s+r_n^2)})^3}e^{-\frac{x_3^2}{4(-s+r_n^2)}}\eta dxds\\
&+C\sum_{i=0}^{n-3}\sum_{k=i+4}^{n-1}\int_{A_k(R)\cap \{|x_3|\leq r_k, -r_{k}^2\leq s\leq -r_{k+1}^2\}}|\pi_{0,i}||u_3|\frac{|x_3|}{(\sqrt{(-s+r_n^2)})^3}e^{-\frac{x_3^2}{4(-s+r_n^2)}}\eta dxds\\
&+C\sum_{i=0}^{n-3}\int_{Q_n(R)}|\pi_{0,i}||u_3|\frac{|x_3|}{(\sqrt{(-s+r_n^2)})^3}\eta dxds\\
\leq &C\sum_{i=0}^{n-3}\sum_{k=i+4}^{n-1}\int_{-r_{k+1}^2}^0r_k^{\frac{1}{p'}}r_i^{\frac{2}{p'}-\frac{3}{\ell}}\|u(\cdot,t)\|_{L^{2\ell}(U_i(R))}^2\|u_3(\cdot,s)\|_{L^p}\frac{1}{(-s+r_n^2)}e^{-\frac{r_k^2}{32(-s+r_n^2)}} ds\\
&+C\sum_{i=0}^{n-3}\sum_{k=i+4}^{n}\left(\int_{-r_{k}^2}^{0}\|u_3(\cdot,s)\|_{L^p}^{\frac{2p}{2p-3}} ds\right)^{\frac{2p-3}{2p}}r_k^{\frac{1}{p'}-2}r_i^{\frac{2}{p'}-\frac{3}{\ell}}\|u\|^2_{L^\frac{4p}{3}(-r_{k}^2,0; L^{2\ell}(U_i(R)) )}\\
\leq &C\sum_{i=0}^{n-3}\sum_{k=i+4}^{n}\int_{-r_{k+1}^2}^0r_k^{\frac{1}{p'}}r_i^{\frac{2}{p'}-\frac{3}{\ell}}\|u(\cdot,t)\|_{L^{2\ell}(U_i(R))}^2\|u_3(\cdot,s)\|_{L^p}
\frac{1}{(-s+r_n^2)}e^{-\frac{r_k^2}{32(-s+r_n^2)}} ds\\
&+C\sum_{i=0}^{n-3}\sum_{k=i+4}^{n}\left(\int_{-r_{k}^2}^{0}\|u_3(\cdot,s)\|_{L^p}^{\frac{2p}{2p-3}} ds\right)^{\frac{2p-3}{2p}}\cdot r_k^{\frac{1}{p'}-2}r_i^{\frac{2}{p'}-\frac{3}{\ell}}r_k^{\frac3p-3+\frac{3}{\ell}}\|u\|^2_{L^\frac{4\ell}{3\ell-3}(-r_{k}^2,0; L^{2\ell}(U_i(R)) )}\\
=&II_{1}'+II_{2}',
\end{align*}
where $1<\ell<p'$. Note that
\begin{align*}
II_{2}'
\leq &C\sum_{i=0}^{n-3}\sum_{k=i+4}^{n}\left(\int_{-r_{k}^2}^{0}\|u_3(\cdot,s)\|_{L^p}^{\tilde{q}}ds\right)^{\frac{1}{\tilde{q}}}
r_k^{2-\frac3p-\frac{2}{\tilde{q}}}r_k^{\frac{1}{p'}-2}r_i^{\frac{2}{p'}-\frac{3}{\ell}}r_k^{\frac3p-3+\frac{3}{\ell}}
E_i(R)\\
\leq &C\sum_{i=0}^{n-3}r_i^{\frac{2}{p'}-\frac{3}{\ell}}\left(\int_{-r_{i}^2}^{0}\|u_3(\cdot,s)\|_{L^p}^{\tilde{q}}ds\right)^{\frac{1}{\tilde{q}}}
\sum_{k=i+4}^{n}r_k^{2-\frac3p-\frac{2}{\tilde{q}}}r_k^{\frac{1}{p'}-2}
r_k^{\frac3p-3+\frac{3}{\ell}}
E_i(R),
\end{align*}
where we choose $\tilde{q}$ close to $q$ and $\ell$ close to $1$ such that
\beno
2-\frac3p-\frac{2}{\tilde{q}}+\frac{1}{p'}-2+\frac3p-3+\frac{3}{\ell}>0.
\eeno
Consequently, we have
\beno
II_{2}'\leq C\sum_{i=0}^{n-3}r_i^{1-\frac3p-\frac{2}{\tilde{q}}-1}\left(\int_{-r_{i+4}^2}^{0}\|u_3(\cdot,s)\|_{L^p}^{\tilde{q}}ds\right)^{\frac{1}{\tilde{q}}}E_i(R).
\eeno
Thanks to
\beno
\frac{1}{(-s+r_n^2)}e^{-\frac{r_k^2}{32(-s+r_n^2)}}\leq C r_k^{-2},
\eeno
it follows from H\"{o}lder inequality that
\begin{align*}
II_{1}'\leq &C\sum_{i=0}^{n-3}\sum_{k=i+4}^{n}\left(\int_{-r_{k}^2}^{0}\|u_3(\cdot,s)\|_{L^p}^{\frac{2p}{2p-3}}ds\right)^{\frac{2p-3}{2p}}
r_k^{\frac{1}{p'}-2}r_i^{\frac{2}{p'}-\frac{3}{\ell}}r_k^{\frac3p-3+\frac{3}{\ell}}
E_i(R)\\
\leq &C\sum_{i=0}^{n-3}r_i^{\frac{2}{p'}-\frac{3}{\ell}}\left(\int_{-r_{i+4}^2}^{0}\|u_3(\cdot,s)\|_{L^p}^{\tilde{q}}ds\right)^{\frac{1}{\tilde{q}}}\sum_{k=i+4}^{n}r_k^{1-\frac3p-\frac{2}{\tilde{q}}+1}r_k^{\frac{1}{p'}-2}
r_k^{\frac3p-3+\frac{3}{\ell}}
E_i(R).
\end{align*}
Hence, we obtain
\beno
II_{1}'\leq C\sum_{i=0}^{n-3}r_i^{1-\frac3p-\frac{2}{\tilde{q}}-1}\left(\int_{-r_{i+4}^2}^{0}\|u_3(\cdot,s)\|_{L^p}^{\tilde{q}}ds\right)^{\frac{1}{\tilde{q}}}E_i(R).
\eeno
Therefore, from above two estimates, we deduce that
\begin{align}\label{eq:pressure-II23}
II_{23}\leq C\sum_{i=0}^{n-3}r_i^{1-\frac3p-\frac{2}{\tilde{q}}-1}\left(\int_{-r_{i+4}^2}^{0}\|u_3(\cdot,s)\|_{L^p}^{\tilde{q}}ds\right)^{\frac{1}{\tilde{q}}}E_i(R).
\end{align}

We denote
\begin{align}\label{eq:def-Ci}
C_i=B_i+r_i^{1-\frac3p-\frac{2}{\tilde{q}}}\left(\int_{-r_{i}^2}^{0}\|u_3(\cdot,s)\|_{L^p}^{\tilde{q}}ds\right)^{\frac{1}{\tilde{q}}},
\end{align}
which combined with \eqref{eq:pressure-II1}, \eqref{eq:pressure-II2122} and \eqref{eq:pressure-II23} implies that
\begin{align*}
\int_{-1}^t\int_{U_0(R)}u_3\pi_0(\partial_3\Phi_n\eta\psi) dxds\leq C\sum_{i=0}^n (r_i^{-1} E_i(R))C_i.
\end{align*}
The bounds of $C_i$ is guaranteed by $B_i$ and Lemma \ref{lem:lorentz}.
\end{proof}

\begin{lemma}\label{prop:pi-02}
Let $(u,\pi)$ be a suitable weak solution of (\ref{eq:NS}) in $\R^3\times (-1,0)$. Suppose that $(u,\pi)$ satisfies the same assumption of Theorem \ref{thm:main2}.  Then we have 
\beno
\Big|\int_{-1}^t\int_{U_0(R)}\pi_0  \Phi_n u\cdot \nabla (\eta\psi) dxds\Big|\leq C\mathcal{E}^{\frac32}+C(R-\rho)^{-1}\mathcal{E}^{\frac12} \sum_{i=0}^{n}r_i^{\frac12}(r_i^{-1}E_i(R)),
\eeno
and
\beno
\Big|\int_{-1}^t\int_{U_0(R)}\nabla \pi_h \cdot u(\Phi_n\eta\psi) dxds\Big|\leq \frac{C}{R-\rho}\mathcal{E}^{\frac32}.
\eeno
\end{lemma}
The proof of this lemma is similar as in \cite{CW}, and the only difference is that we need to deal with a large value of $R$.
\begin{proof}
We first notice that
\begin{align*}
&\left|\int_{-1}^t\int_{U_0(R)}\pi_0  \Phi_n u\cdot \nabla (\eta\psi) dxds\right|\\
&=\left|\int_{-1}^t\int_{U_0(R)}\pi_0  \Phi_n u\cdot \nabla \eta(\psi) dxds+ \int_{-1}^t\int_{U_0(R)}\pi_0  \Phi_n u\cdot \nabla (\psi)\eta dxds\right|.
\end{align*}
Here we give the estimates about the second term on the right side of the above inequality. We find that
\begin{align*}
\int_{-1}^t&\int_{U_0(R)}\pi_0 \Phi_n\eta  u\cdot \nabla \psi dxds\\
= &\sum_{k=0}^{n}\int_{-1}^t\int_{U_0(R)}\pi_0 \Phi_n\eta  u\cdot \nabla \psi \phi_k dxds
= \sum_{j=0}^{n}\sum_{k=0}^{n}\int_{-1}^t\int_{U_0(R)}\pi_{0,j} \Phi_n\eta  u\cdot \nabla \psi \phi_k dxds\\
=&\sum_{k=0}^{n}\sum_{j=k}^{n}\int_{-1}^t\int_{U_0(R)}\pi_{0,j} \Phi_n\eta  u\cdot \nabla \psi \phi_k dxds+\sum_{j=0}^{n}\sum_{k=j+1}^{n}\int_{-1}^t\int_{U_0(R)}\pi_{0,j} \Phi_n\eta  u\cdot \nabla \psi \phi_kdxds\\
=&I'+II'.
\end{align*}
As in the proof of Lemma \ref{prop:pi-01}, we know that $I'$ can be written as
\beno
I'=\sum_{k=0}^{n}\int_{-1}^t\int_{U_0(R)}\Pi_{0,k} \Phi_n\eta  u\cdot \nabla \psi \phi_k dxds
\eeno
with
\beno
\Pi_{0,k}=\left\{\begin{array}{lll}
\pi_0,\quad {\rm if }\quad k=0;\\
T(\chi_{k} f),\quad {\rm if }\quad k=1,\cdots,n.
\end{array}\right.
\eeno
Then we get
\begin{align*}
I'\leq &\frac{C}{R-\rho}
\sum_{i=0}^{n}\int_{A_i(R)}|\Pi_{0,i}||u|\frac{1}{\sqrt{(-s+r_n^2)}}e^{-\frac{x_3^2}{4(-s+r_n^2)}}\eta dxds\\
\leq &\frac{C}{R-\rho}\sum_{i=0}^{n}r_i^{-1}r_i^{\frac12}\|u\|^2_{L^\frac{10}{3}(-r_{i}^2,0; L^{\frac{10}{3}}(U_i(R)) )}\|u\|_{L^\frac{20}{3}(-r_{i}^2,0; L^{\frac{5}{2}}(U_i(R)) )}\\
\leq&C(R-\rho)^{-1}\mathcal{E}^{\frac12} \sum_{i=0}^{n}r_i^{\frac12}(r_i^{-1}E_i(R)).
\end{align*}
By a similar argument as $II_2$ in the proof of  Lemma \ref{prop:pi-01}, we can obtain the estimate of $II'$ by using the norm of $\|u\|_{L_t^{\frac{4p}{3p-6}}L_x^p}$ instead of the norm of $u_3$. Then
\beno
\Big|\int_{-1}^t\int_{U_0(R)}\pi_0  \Phi_n u\cdot \nabla (\eta\psi) dxds\Big|\leq C\mathcal{E}^{\frac32}+C(R-\rho)^{-1}\mathcal{E}^{\frac12} \sum_{i=0}^{n}r_i^{\frac12}(r_i^{-1}E_i(R)).
\eeno

Now we turn to show the proof of the second result of the lemma. We first choose a cut-off function $\zeta(x_3,t)\in C_c^\infty\big((-\frac12,\frac12)\times(-\frac14,0]\big)$ satisfying $\zeta(x_3,t)=1$ in $(-\frac14,\frac14)\times(-\frac{1}{16},0]$  and
\beno
 |\partial_3\zeta|\leq C.
\eeno
Then
\beno
&&\left|\int_{Q_0(R)}\nabla \pi_h \cdot u(\Phi_n\eta\psi)(1-\zeta) dxds\right|\\
&&\leq \left|\int_{Q_0(R)}\pi_h  u\cdot\nabla (\Phi_n\eta\psi(1-\zeta)) dxds\right|\\
&&\leq \frac{C}{R-\rho}\|\pi_h\|_{L^{\frac32}(Q_0(R))}\|u\|_{L^{3}(Q_0(R))}\leq \frac{C}{R-\rho}\mathcal{E}^{\frac32}.
\eeno
Moreover, for fixed $R$ and $\rho$, there exist finite balls centered at $x'_j\in B'(\frac{R+\rho}{2})$ with $j=1,\cdots,J$, whose radius is $\frac{R-\rho}{4}$, and there hold
\beno
&&\bigcup_{j=1}^J\Big\{x', |x'-x'_j|<\frac{R-\rho}{4}\Big\}\supset B'\Big(\frac{R+\rho}{2}\Big),\\
&&\sum_{j=1}^J\Big|\Big\{x', |x'-x'_j|<\frac{R-\rho}{2}\Big\}\Big|\leq C |B'(R)|.
\eeno
Then
\begin{align*}
&\left|\int_{-1}^t\int_{U_0(R)}\nabla \pi_h \cdot u(\Phi_n\eta\psi)\zeta dxds\right|\\
&\leq C \sum_{k=1}^n\sum_{j=1}^Jr_k^{-1}\int_{Q_k(R)\cap\{|x'-x'_j|<\frac{R-\rho}{4}\}}|\nabla \pi_h||u|dxds\\
&\leq C \sum_{k=1}^n\sum_{j=1}^Jr_k^{-1}\|\nabla \pi_h\|_{L^\frac32(-r_k^2,0; L^\infty(U_k(R)\cap\{|x'-x'_j|<\frac{R-\rho}{4}\}))}\|u\|_{L^3(-r_k^2,0; L^1(U_k(R)\cap\{|x'-x'_j|<\frac{R-\rho}{4}\}))}\\
&\leq C (R-\rho)^{\frac43}\sum_{k=1}^n\sum_{j=1}^Jr_k^{-1/3}\|\nabla \pi_h\|_{L^\frac32(-r_k^2,0; L^\infty(U_k(R)\cap\{|x'-x'_j|<\frac{R-\rho}{4}\}))}\cdot\\
&\quad\cdot\|u\|_{L^3(-r_k^2,0; L^3(U_k(R)\cap\{|x'-x'_j|<\frac{R-\rho}{4}\}))}\\
&\leq C(R-\rho)^2 \sum_{k=1}^n\sum_{j=1}^Jr_k^{-1/3}\|\nabla \pi_h\|^\frac32_{L^\frac32(-r_k^2,0; L^\infty(U_k(R)\cap\{|x'-x'_j|<\frac{R-\rho}{4}\})}\\
&\quad+C \sum_{k=1}^nr_k^{-1/3}\|u\|^3_{L^3(-r_k^2,0; L^3(U_k(R)))}.
\end{align*}
For any 
\beno
x^*\in U_k(R)\cap\Big\{x=(x',x_3);|x'-x'_j|<\frac{R-\rho}{4}\Big\},\eeno
we have  
\beno
d(x^*,\partial U_0(R))>\min\Big\{\frac12,\frac{R-\rho}{4}\Big\}=\frac{R-\rho}{4}
\eeno
due to $k\geq 1$ and $|R-\rho|\leq \frac12.$ Thus, there exists $x_3^*\in (-\frac12,\frac12)$ such that 
\beno
x^*\in B\Big((x_j',x_3^*); \frac{R-\rho}{2}\Big)\subset U_0(R)\cap\Big\{|x'-x'_j|<\frac{R-\rho}{2}\Big\}.
\eeno 
Since $\pi_h$ is harmonic in $U_0(R)$, 
there holds
\begin{align*}
|\nabla \pi_h|(x^*)\leq& \frac{C}{R-\rho}\frac{1}{|R-\rho|^3}\int_{B((x_j',x_3^*); \frac{R-\rho}{2})}|\pi_h|dx\\
\leq& \frac{C}{(R-\rho)^3}\|\pi_h\|_{L^\frac32(U_0(R)\cap\{|x'-x'_j|<\frac{R-\rho}{2}\})},
\end{align*}
which implies 
\beno
&&\|\nabla \pi_h\|^\frac32_{L^\frac32(-r_k^2,0; L^\infty(U_k(R)\cap\{|x'-x'_j|<\frac{R-\rho}{4}\})}\\
&&\leq \int_{-r_k^2}^0\|\nabla \pi_h(\cdot,s)\|_{L^\infty( U_k(R)\cap\{|x'-x'_j|<\frac{R-\rho}{4}\})}^{\frac32}ds\\
&&\leq \frac{C}{(R-\rho)^3}\int_{-r_k^2}^0\int_{(U_0(R)\cap\{|x'-x'_j|<\frac{R-\rho}{2}\})}|\pi_h|^\frac32ds.
\eeno
Hence,
\begin{align*}
I_{33}\leq&
\frac{C}{R-\rho}\sum_{k=1}^nr_k^{-1/3}\left(\| \pi_h\|^\frac32_{L^\frac32(-r_k^2,0; L^\frac32(U_0(R))}+\|u\|^3_{L^3(-r_k^2,0; L^3(U_k(R))}\right)\\
\leq&
\frac{C}{R-\rho}\sum_{k=1}^nr_k^{1/6}\left(\|\pi_h\|^\frac32_{L^2(-r_k^2,0; L^\frac32(U_0(R))}+\|u\|^3_{L^4(-r_k^2,0; L^3(U_k(R))}\right)\\
\leq& \frac{C}{R-\rho}\Big(\|\pi\|^\frac32_{L^2(-1,0;L^\frac32(U_0(R))}+\mathcal{E}^{\frac32}\Big).
\end{align*}
Applying Calder\'on-Zygmund estimates, there holds
\beno
\Big|\int_{-1}^t\int_{U_0(R)}\nabla \pi_h \cdot u(\Phi_n\eta\psi) dxds\Big|\leq \frac{C}{R-\rho}\mathcal{E}^{\frac32}.
\eeno

The proof is completed.
\end{proof}

\section{Proof of Theorem \ref{thm:main2}}

This section is devoted to the proof of Theorem \ref{thm:main2} by using Lemma \ref{lem:iter} and the interpolation inequality.

\begin{proof}
Gathering the estimates in Lemma  \ref{prop:I1I2}, \ref{prop:pi-01} and \ref{prop:pi-02}, we have
\begin{align*}
r_n^{-1}E_n(\rho)
\leq & C\frac{\mathcal{E}}{(R-\rho)^2}+\frac{C}{R-\rho}\mathcal{E}^{\frac32}+C\sum_{i=0}^{n}(r_i^{-1}E_i(R))(B_i+C_i)\\
&+ C(R-\rho)^{-1}\mathcal{E}^{\frac12} \sum_{i=0}^{n}r_i^{\frac12}(r_i^{-1}E_i(R)).
\end{align*}
Note that  $B_i\leq C_i$  and $r_n^{-1}E_n(R)\leq 2 r_{n-1}^{-1}E_{n-1}(R) $. Then we have
\begin{align*}
r_n^{-1}E_n(\rho)
\leq & C\frac{\mathcal{E}}{(R-\rho)^2}+\frac{C}{R-\rho}\mathcal{E}^{\frac32}+C\sum_{i=0}^{n-1}(r_i^{-1}E_i(R))C_i\\
&+ C(R-\rho)^{-1}\mathcal{E}^{\frac12} \sum_{i=0}^{n-1}r_i^{\frac12}(r_i^{-1}E_i(R)).
\end{align*}
Choosing $\rho=R-\frac12$, we have
\begin{align*}
r_n^{-1}E_n(\rho)
\leq & C\mathcal{E}+C\mathcal{E}^{\frac32}+C\sum_{i=0}^{n-1}(r_i^{-1}E_i(R))C_i\\
&+ C\mathcal{E}^{\frac12} \sum_{i=0}^{n-1}r_i^{\frac12}(r_i^{-1}E_i(R)).
\end{align*}
Then let $R,\rho\rightarrow\infty$, which deduces that
\beno
r_n^{-1}E_n(\infty)
\leq C\mathcal{E}+C\mathcal{E}^{\frac32}+C\sum_{i=0}^{n-1}(r_i^{-1}E_i(\infty))C_i+C\mathcal{E}^{\frac12} \sum_{i=0}^{n-1}r_i^{\frac12}(r_i^{-1}E_i(\infty))
\eeno
At last, due to $\sum_{i\geq 0}C_i\leq C\|u_3\|_{L^{q,1}_tL^p_x}$ and Lemma \ref{lem:iter}, we have for any $n\in\mathbb{N}$,
\beno
r_n^{-1}E_n(\infty)\leq C(\mathcal{E}+\mathcal{E}^{\frac{3}{2}})e^{\sum_{i=0}^\infty (C_i+\mathcal{E}^{\frac12} r_i^{\frac12}) }\leq C(\mathcal{E}+\mathcal{E}^{\frac{3}{2}})e^{C\mathcal{E}^{\frac12}},
\eeno
which yields that
\begin{align*}
r^{-2}\|u\|_{L^3(Q_r)}^3\leq Cr^{-\frac{3}{2}}\|u\|^3_{L^4(-r^2,0;L^3(B(r)))}\leq C (r^{-1}E(r))^{\frac{3}{2}}\leq C.
\end{align*}
The proof is completed.
\end{proof}

\appendix
\section{}
\begin{lemma}\label{lem:iter}
Let $\{b_j\}_{j\in\mathbb{N}}$, $\{y_j\}_{j\in\mathbb{N}}$ be non-negative series and satisfy the following inequality
\beno
y_n\leq C_0+\sum_{j=0}^{n-1}b_jy_j,  n\geq 1~~\mathrm{and}~~y_0\leq C_0.
\eeno
Then we have for any $n\in\mathbb{N}$,
\beno
y_n\leq C_0e^{\sum_{j=0}^{n-1}b_j}.
\eeno
\end{lemma}
\begin{proof}
We first define the following non-negative series $\{x_j\}$
\begin{align*}
x_0=C_0,\quad  x_n=C_0+\sum_{j=0}^{n-1}b_jx_j,\quad n\geq 1.
\end{align*}
It is easy to check that for any $j\in\mathbb{N}$, $x_j\geq y_j$. On the other hand, by the definition of $\{x_j\}$,  it can be represented as for any $n\geq 1$
\begin{align*}
x_n=C_0\prod_{i=0}^{n-1}(1+b_i)\leq C_0 e^{\sum_{i=0}^{n-1}b_i}.
\end{align*}
Hence, we obtain that for any $n\geq 1$
\begin{align*}
y_n\leq x_n\leq C_0 e^{\sum_{i=0}^{n-1}b_i},
\end{align*}
which along with the condition that $y_0\leq C_0$ completes the proof of this lemma.
\end{proof}

\begin{lemma}\label{lem:at-decom}
Let $0<p,q<\infty$. Then for any $f\in L^{p,q}(\mathbb{R})$, there exists a sequence $\{c_n\}_{n\in\mathbb{Z}}\in \ell^q$ and sequence of functions $\{f_n\}_{n\in\mathbb{Z}}$ with each $f_n$ bounded by $2^{-n/p}$ and supported on a set of measure $2^n$ such that
\begin{align*}
f=\sum_{n\in\mathbb{Z}}c_nf_n,
\end{align*}
and
\begin{align*}
c(p,q)\|\{c_n\}\|_{\ell^q}\leq \|f\|_{L^{p,q}}\leq C(p,q)\|\{c_n\}\|_{\ell^q},
\end{align*}
where the constant $c(p,q)$ and $C(p,q)$ only depend on $p,q$.
\end{lemma}
\begin{proof}
Let $f\in L^{p,q}(\mathbb{R})$. We denote $f^*$ as the corresponding decreasing rearrangement of $f$. We let
\begin{align}
c_n:=2^{n/p}f^*(2^n),\quad A_n:=\{x: f^*(2^{n+1})<|f(x)|\leq f^*(2^n)\}\quad\mathrm{and}\quad f_n:=c_n^{-1}f\mathbf{1}_{A_n}.
\end{align}
By direct calculation, it is easy to check that
\begin{align*}
f=\sum_{n\in\mathbb{Z}}c_nf_n.
\end{align*}
Now we start to prove the second statement. We notice that by the definition of Lorentz space, we have
\begin{align*}
\|f\|_{L^{p,q}}^q&= \sum_{n\in\mathbb{Z}}\int_{2^n}^{2^{n+1}}(s^{1/p}f^*(s))^qs^{-1}ds\\
&\leq\sum_{n\in\mathbb{Z}}(f^{*}(2^n))^q2^{nq/p}2^{-nq/p}\int_{2^n}^{2^{n+1}}s^{\frac{q}{p}-1}ds\\
&\leq \frac{p}{q}(2^{q/p}-1) \sum_{n\in\mathbb{Z}}(f^{*}(2^n))^q2^{nq/p}=\frac{p}{q}(2^{q/p}-1) \|(c_n)\|^q_{\ell^q}
\end{align*}
and
\begin{align*}
\|f\|_{L^{p,q}}^q&= \sum_{n\in\mathbb{Z}}\int_{2^n}^{2^{n+1}}(s^{1/p}f^*(s))^qs^{-1}ds\\
&\geq \sum_{n\in\mathbb{Z}}2^{q(n+1)/p}(f^*(2^{n+1}))^q2^{-q(n+1)/p}\int_{2^n}^{2^{n+1}}s^{q/p-1}ds\\
&=\frac{p}{q}(1-2^{-q/p})\|(c_n)\|_{\ell^q}^q.
\end{align*}
\end{proof}

\begin{lemma}\label{lem:lorentz}
 For any \beno
\frac{2p}{2p-3}\leq \tilde{q}<q=\frac{2p}{p-3}, \quad 3<p<\infty,
\eeno
we have
\beno
\sum_{k=0}^{\infty}{r_k}^{1-\frac3p-\frac{2}{\tilde{q}}}\left(\int_{-r_{k}^2}^{0}\|u_3(\cdot,s)\|_{L^p}^{\tilde{q}}ds\right)^{\frac{1}{\tilde{q}}}\leq C\|u_3\|_{L^{q,1}_tL^p_x(-r_0,0;\mathbb{R}^3)}
\eeno
\end{lemma}
\begin{proof}
Let $f(s)=\|u_3(\cdot,s)\|_{L^p}$. By Lemma \ref{lem:at-decom}, we know that \beno
f=\sum_{\ell=0}^{+\infty}c_\ell f_\ell,\quad \|f\|_{L^{q,1}}\approx \sum_{\ell=0}^\infty |c_\ell|,
\eeno
where
\beno
|f_\ell|\leq 2^{\frac{\ell}{q}},\quad |D_\ell={\rm supp}~ f_\ell|\approx2^{-\ell}.
\eeno
Then we have
\beno
&&\sum_{k=0}^{\infty}{r_k}^{1-\frac3p-\frac{2}{\tilde{q}}}\left(\int_{-r_{k}^2}^{0}\|u_3(\cdot,s)\|_{L^p}^{\tilde{q}}ds\right)^{\frac{1}{\tilde{q}}}
=\sum_{k=0}^{\infty}{r_k}^{1-\frac3p-\frac{2}{\tilde{q}}}\left(\int_{I_k}|f|^{\tilde{q}}ds\right)^{\frac{1}{\tilde{q}}}\\
&&\leq \sum_{k=0}^{\infty}{r_k}^{1-\frac3p-\frac{2}{\tilde{q}}}  \sum_\ell |c_\ell|  2^{\frac{\ell}{q}}|D_\ell\cap I_k|^{\frac{1}{\tilde{q}}}
\leq   \sum_\ell |c_\ell|  2^{\frac{\ell}{q}}\sum_{k=0}^{\infty}{r_k}^{1-\frac3p-\frac{2}{\tilde{q}}}|D_\ell\cap I_k|^{\frac{1}{\tilde{q}}}\\
&&\leq   \sum_\ell |c_\ell|  2^{\frac{\ell}{q}}\sum_{k=0}^{\ell/2}{r_k}^{1-\frac3p-\frac{2}{\tilde{q}}}|D_\ell\cap I_k|^{\frac{1}{\tilde{q}}}
+\sum_\ell |c_\ell|  2^{\frac{\ell}{q}}\sum_{k=\ell/2}^{\infty}{r_k}^{1-\frac3p-\frac{2}{\tilde{q}}}|D_\ell\cap I_k|^{\frac{1}{\tilde{q}}},
\eeno
where $I_k=(-r_{k}^2,0)$.
On the other hand, we notice that for any $k\leq\ell/2$,
\beno
{r_k}^{1-\frac3p-\frac{2}{\tilde{q}}}|D_\ell\cap I_k|^{\frac{1}{\tilde{q}}}\leq C2^{-\frac{\ell}{\tilde{q}}}2^{-k(1-\frac{3}{p}-\frac{2}{\tilde{q}})},
\eeno
and for any $\ell/2\leq k<\infty$,
\beno
{r_k}^{1-\frac3p-\frac{2}{\tilde{q}}}|D_\ell\cap I_k|^{\frac{1}{\tilde{q}}}\leq C2^{-\frac{2k}{\tilde{q}}}2^{-k(1-\frac{3}{p}-\frac{2}{\tilde{q}})}=C2^{-k(1-\frac{3}{p})}.
\eeno
Then we have
\beno
&&\sum_{k=0}^{\infty}{r_k}^{1-\frac3p-\frac{2}{\tilde{q}}}\left(\int_{-r_{k}^2}^{0}\|u_3(\cdot,s)\|_{L^p}^{\tilde{q}}ds\right)^{\frac{1}{\tilde{q}}}\\
&&\leq  \sum_{\ell=0}^\infty|c_\ell|2^{\frac{\ell}{q}-\frac{\ell}{\tilde{q}}}\sum_{k=0}^{\ell/2}2^{-k(1-\frac{3}{p}-\frac{2}{\tilde{q}})}+\sum_{\ell=0}^\infty|c_\ell|2^{\frac{\ell}{q}}\sum_{k=\ell/2}^{\infty}2^{-k(1-\frac{3}{p})},
\eeno
which along with the restriction on $p,q$ implies that
\begin{align*}
\sum_{k=0}^{\infty}{r_k}^{1-\frac3p-\frac{2}{\tilde{q}}}\left(\int_{-r_{k}^2}^{0}\|u_3(\cdot,s)\|_{L^p}^{\tilde{q}}ds\right)^{\frac{1}{\tilde{q}}}\leq C\sum_{l=0}^{\infty}|c_l|\leq C\|u_3\|_{L^{q,1}_tL^p_x(-r_0,0;\mathbb{R}^3)}.
\end{align*}

The proof is completed.
\end{proof}

\smallskip

Finally let us recall the lemma about the harmonic function in \cite{CW}.

\begin{lemma}\label{lem:harmonic estimate}
Let $0<r\leq R<\infty$ and $h:B'(2R)\times (-r,r)\rightarrow \mathbb{R}$ be harmonic. Then for all
$0<\rho\leq \frac{r}{4}$ and $1\leq \ell\leq p<\infty$,
\beno
\|h\|_{L^p(B'(R)\times (-\rho,\rho))}^p\leq c \rho r^{2-3\frac{p}{\ell}}\|h\|_{L^\ell(B'(2R)\times (-r,r))}^p.
\eeno
\end{lemma}

\section*{Acknowledgments}

W. Wang was supported by NSFC under grant 11671067 and the Fundamental Research Funds for the Central Universities. Z. Zhang is partially supported by NSF of China under Grant 11425103.

\end{document}